\documentclass[reqno,draft]{amsart}
\usepackage{amsmath, amsthm, amssymb, dsfont}
\usepackage[a4paper]{geometry}					
\usepackage{graphicx}
\usepackage{tikz}
\usepackage{capt-of}
\usepackage{enumerate}

\usepackage{latexsym}
\usepackage{bbm}
\usepackage{mathtools}
\usepackage[a]{esvect}

\usepackage{comma}
\usepackage{cutwin}
\usepackage{pstricks}
\usepackage{pst-plot}
\usepackage{pgfplots}
\usepackage{float}

\numberwithin{equation}{section}
\newcounter{assumptions}

\usepackage{changes}

\theoremstyle{plain}
\newtheorem{theorem}{Theorem}[section]

\theoremstyle{remark}
\newtheorem{remark}[theorem]{Remark}
\theoremstyle{definition}

\newtheorem{example}[theorem]{Example}

\newcommand{\N}{\mathbb{N}}

\newcommand{\R}{\mathbb{R}}

\newcommand{\C}{\mathbb{C}}

\DeclareMathOperator{\Real}{\mathrm{Re}}

\newcommand{\A}{\mathcal{A}}

\newcommand{\F}{\mathcal{F}}

\newcommand{\cC}{\mathcal{C}}

\newcommand{\cL}{\mathcal{L}}

\newcommand{\Prob}{\mathbf{P}}

\newcommand{\E}{\mathbf{E}}

\newcommand{\eqdist}{%
  \mathrel{\vbox{\offinterlineskip\ialign{%
    \hfil##\hfil\cr
    $\scriptscriptstyle\mathrm{d}$\cr
    \noalign{\kern.1ex}
    $=$\cr
}}}}

\newcommand{\1}{\mathbbm{1}}

\newcommand{\interior}[1]{%
	{\kern0pt#1}^{\mathrm{o}}%
}

\newcommand{\distto}{%
  \mathrel{\vbox{\offinterlineskip\ialign{%
    \hfil##\hfil\cr
    $\scriptscriptstyle\mathrm{d}$\cr
    \noalign{\kern-.05ex}
    $\to$\cr
}}}}
\newcommand{\Probto}{%
  \mathrel{\vbox{\offinterlineskip\ialign{%
    \hfil##\hfil\cr
    $\scriptscriptstyle\Prob$\cr
    \noalign{\kern-.05ex}
    $\to$\cr
}}}}
\newcommand{\TVto}{%
  \mathrel{\vbox{\offinterlineskip\ialign{%
    \hfil##\hfil\cr
    $\scriptscriptstyle\mathrm{TV}$\cr
    \noalign{\kern-.05ex}
    $\to$\cr
}}}}

\newcommand{\ds}{\mathrm{d} \mathit{s}}

\newcommand{\dx}{\mathrm{d} \mathit{x}}

\newcommand{\transp}{\mathsf{T}}

\newcommand{\bv}{\mathbf{v}}
\newcommand{\e}{\mathsf{e}}
\newcommand{\f}{\mathsf{f}}

\newcommand{\bxi}{\boldsymbol{\xi}}

\newcommand{\vA}{\vec{\mathbf{A}}}
\newcommand{\adj}{\mathrm{adj}}

\newcommand{\bmu}{\boldsymbol{\mu}}

\newcommand{\Ip}{\mathbf{I}_{p}}

\newcommand{\Ik}{\mathbf{I}_{\klambda}}

\newcommand{\defeq}{\vcentcolon=}

\newcommand{\Gen}{\mathcal{G}}
\newcommand{\familytree}{\mathcal{T}}
\newcommand{\I}{\mathcal{I}}

\newcommand{\Surv}{\mathcal{S}}
\newcommand{\p}{[ \, p \,]}
\newcommand{\Dom}{\mathcal{D}}
\newcommand{\cZ}{\mathcal{Z}}

\newcommand{\phil}{\phi_{\lambda}}
\newcommand{\philu}{\phi_{\lambda,u}}

\newcommand{\klambda}{\mathtt{k}_\lambda}

\newcommand{\HS}{\mathsf{HS}}

\usepackage{amsmath} 
\usepackage{array}   

\begin{document}

\title[Martingales in multi-type general branching processes]{Convergence of complex martingales in supercritical multi-type general branching processes in $L^q$ for $1 < q \leq 2$}
\author{Konrad Kolesko \and Matthias Meiners \and Ivana Tomic}

\begin{abstract}
Nerman's martingale plays a central role in the law of large numbers for both, single- and multi-type,
supercritical general branching processes.
There are further, complex-valued Ner\-man-type martingales in the single-type process
that figure in the finer fluctuations of these processes.
We construct the analogous martingales for the process with finitely many types
and give sufficient conditions for these martingales to converge in $L^q$ for $q \in (1,2]$.
\smallskip

\\{\bf Subclass:} MSC: 60J80
\end{abstract}

\maketitle

\section{Introduction}

Consider a Crump-Mode-Jagers branching process in which each individual is assigned one of finitely many types,
labeled $1, \ldots, p$. Denote by $\cZ_t^j$ the number of type-$j$ individuals born up to and including time $t \geq 0$.
The dynamics of the process are governed by a matrix of offspring point processes $\bxi = (\xi^{i,j})_{i,j=1,\ldots,p}$
where $\xi^{i,j}$ is the point process of birth times of type-$j$ offspring produced by an ancestor of type $i$.
Suppose that the process is supercritical and irreducible
and that there exists a Malthusian parameter $\alpha > 0$ such that the dominant (Perron-Frobenius)
eigenvalue of the matrix $\cL\bmu(\alpha)$ of Laplace transforms of the intensity measures $\mu^{i,j} \coloneqq \E[\xi^{i,j}]$
evaluated at $\alpha$ equals $1$.
Under this condition, one can construct a nonnegative martingale \((W_t(\alpha))_{t \geq 0}\),
known as Nerman’s martingale, where
\begin{equation*}
W_t(\alpha) = \sum_{u \in \cC_t} \frac{v_{\tau(u)}}{v_i}e^{-\alpha S(u)},
\end{equation*}
with $\cC_t$ denoting the set of individuals who are the first in their line of descent to be born strictly after time $t$,
$\bv = (v_1,\ldots,v_p)^\transp$ is the unique normalized  nonnegative right Perron-Frobenius eigenvalue of $\cL\bmu(\alpha)$
and $S(u)$ and $\tau(u)$ representing the birth time and type, respectively, of individual $u$. The following law of large numbers holds:
\begin{equation*}
e^{-\alpha t} \cZ_t^j \to c_{i,j} W(\alpha)
\end{equation*}
as $t \to \infty$ where $c_{i,j} > 0$ is a constant depending on the types $i,j$
and $W(\alpha) = \lim_{t \to \infty} W_t(\alpha)$ is the limit of Nerman's martingale.
Depending on the specific additional assumptions imposed, this convergence may hold
in $\Prob^i$-probability, in $L^1(\Prob^i)$, or $\Prob^i$-almost surely,
where $\Prob^i$ denotes the law of the process with a type-$i$ ancestor.
These results also extend to more general processes---specifically, general branching processes counted with a random characteristic.
For further details, see \cite{Iksanov+Meiners:2015,Jagers:1989,Jagers+Nerman:1984,Nerman:1979,Nerman:1981,Olofsson:2009}.

A refined result, also valid for general branching processes counted with a random characteristic,
has been obtained recently for the single-type case \cite{Iksanov+al:2024}, in which we write $\xi$ for $\xi^{1,1}$.
In this paper, the finiteness of the Laplace transform $\cL\mu$ of $\E[\xi]$ at $\frac\alpha2$ is assumed.
In this case, the roots $\lambda \in \C$ of the equation $\cL\mu(\lambda)=1$ come into play.
For any such $\lambda$ with $\frac\alpha2 < \lambda \leq \alpha$, there is a Nerman-type martingale $(W_t(\lambda))_{t \geq 0}$
where 
\begin{equation*}
W_t(\lambda) = \sum_{u \in \cC_t} e^{-\lambda S(u)},		\quad	t \geq 0.
\end{equation*}
Under suitable $L^2$-assumptions, these martingales converge in $L^2$.
We write $W(\lambda)$ for the limit of $(W_t(\lambda))_{t \geq 0}$ in $L^2$.
Assuming, for simplicity of presentation, that there is no root $\lambda$ satisfying $\Real(\lambda)=\frac\alpha2$
and that all roots $\lambda$ with $\frac\alpha2 < \Real(\lambda) \leq \alpha$ are simple,
one has
\begin{equation*}
e^{-\frac\alpha2 t} \bigg(\cZ_t^\varphi - c_\alpha e^{\alpha t} W(\alpha)
- \sum_{\substack{\frac\alpha2 < \Real \lambda < \alpha:\\ \cL\mu(\lambda)=1}} c_\lambda e^{\lambda t} W(\lambda)\bigg)
\distto c \sqrt{W(\alpha)} N
\end{equation*}
where $\cZ_t^\varphi$ denotes a general branching process counted with random characteristic $\varphi$,
$c,c_\alpha, c_\lambda$ are constants (depending on $\varphi$, $\lambda$ and $\mu$),
and $N$ is a standard normal random variable independent of the branching process, see \cite[Theorem 2.9]{Iksanov+al:2024}.

A multi-type extension of this result would be highly desirable, as it is expected to have many applications.
First, we need the multi-type analogues of the martingales $(W_t(\lambda))_{t \geq 0}$.
The multi-type counterpart of the characteristic equation $\cL\mu(\lambda) = 1$ is given by
\begin{equation} \label{eq:det(Lbmu-Ip)=0}
\det(\cL\bmu(\lambda) - \Ip) = 0,
\end{equation}
where $\Ip$ denotes the $p \times p$ identity matrix.  
To establish a corresponding result in the multi-type setting,  
it is necessary to identify all martingales associated with each solution $\lambda$ to Equation~\eqref{eq:det(Lbmu-Ip)=0},  
and to prove their convergence in $L^2$.
This is the main objective of the present paper.

\section{Assumptions and main result}

We continue by rigorously formulating the model, followed by a precise statement of our main results.

\subsection{Model description}	\label{subsec:General branching process}

In this section, we formally introduce the multi-type general (Crump-Mode-Jagers) branching process.
The process begins with a single individual---the \emph{ancestor}---born at time~$0$.
It is assigned the label~$\varnothing$, the empty tuple, and, like every individual in the process,
is of one of finitely many types:
\begin{align*}
\tau(\varnothing) \in \p \coloneqq \{1, \ldots, p\},
\end{align*}
where $p \in \N = \{1,2,3,\ldots\}$ denotes the number of distinct types.
If $\tau(\varnothing) = i$, meaning the ancestor is of type~$i$, then its children of type~$j$ are born at the points of the random point process $\xi^{i,j}$, for $j = 1, \ldots, p$.
Throughout, we assume that each $\xi^{i,j}$ is supported on $[0, \infty)$, and hence its total mass is given by
$N^{i,j} = \xi^{i,j}([0, \infty))$.
The random variables $N^{i,j}$ take values in $\N_0 \cup \{\infty\}$, where $\N_0 \coloneqq \N \cup \{0\}$.  
In particular, we allow for the possibility that $\Prob(N^{i,j} = \infty) > 0$,  
which means that the ancestor may give rise to an infinite number of children of a given type.
We denote by
\begin{equation*}
\bxi = (\xi^{i,j})_{i,j = 1,\ldots,p}
\end{equation*}
the $p \times p$ matrix of point processes that determines the reproduction behavior of the ancestor.  
The individual processes $\xi^{i,j}$ are referred to as \emph{offspring processes},
and the matrix $\bxi$ as the \emph{offspring process matrix}.

Throughout this text, we denote by $\e_1, \ldots, \e_p$ the canonical basis vectors of $\R^{1 \times p}$,
considered as row vectors.  
Using this notation, we express the vector of offspring point processes for an individual of type~$i$ as
\begin{align}
\xi^{i} = \sum_{j=1}^p \xi^{i,j} \, \e_j = (\xi^{i,1}, \ldots, \xi^{i,p}),
\end{align}
which collects, component-wise, the reproduction processes for each offspring type.
Similarly, we define the offspring count vector for an individual of type~$i$ as
\begin{equation*}
N^i \coloneqq \sum_{j=1}^p N^{i,j} \, \e_j.
\end{equation*}
We further define
\begin{align*}
\xi(\cdot) \coloneqq \sum_{j=1}^p \xi^{\tau(\varnothing),j}(\cdot)
= \sum_{k=1}^{N} \delta_{X_k}(\cdot),
\end{align*}
where $N = \sum_{j=1}^p N^{\tau(\varnothing),j} = \xi([0,\infty))$
is the total number of offspring of the ancestor
and $\delta_x$ is the Dirac measure with a point at $x$.
The point process $\xi$ thus records the birth times of all children of the ancestor, irrespective of their types.
Without loss of generality, we assume
\begin{equation*}
0 \leq X_1 \leq X_2 \leq \ldots,
\end{equation*}
in which case each birth time $X_k$ can be recovered as
\begin{equation*}
X_k = \inf\{t \geq 0 : \xi([0, t]) \geq k\}, \quad k \in \N
\end{equation*}
where $X_k = \infty$ indicates that the ancestor has strictly fewer than $k$ children.

The multi-type general branching process is then constructed by allowing each child of the ancestor to produce offspring
according to an independent copy of the point process matrix~$\bxi$.
These offspring, in turn, reproduce independently according to further independent copies of~$\bxi$, and so on.
To formalize this, we introduce the standard Ulam-Harris notation.  
Let $\N^0 \coloneqq \{\varnothing\}$ denote the set containing only the empty tuple.  
The infinite Ulam-Harris tree is defined as
\begin{equation*}
\I \coloneqq \bigcup_{n \in \N_0} \N^n,
\end{equation*}
i.e., the set of all finite sequences (tuples) of natural numbers.
We abbreviate elements $u = (u_1,\ldots,u_n) \in \I$ by writing $u_1\ldots u_n$.  
Given another element $v = v_1 \ldots v_m \in \I$, their concatenation is denoted $uv = u_1 \ldots u_n v_1 \ldots v_m$.
The Ulam-Harris tree $\I$ serves as the set of labels for all potential individuals.  
Each $u \in \I$ identifies a unique individual with ancestral line
\begin{align*}
\varnothing \rightarrow u_1 \rightarrow u_1u_2 \rightarrow \cdots \rightarrow u_1\ldots u_n = u,
\end{align*}
where $u_1$ is the $u_1$-th child of the ancestor, $u_1u_2$ the $u_2$-th child of $u_1$, and so on.
We write $|u| = n$ to indicate that $u$ belongs to generation $n$.  
For $u \in \I$ and $i \in \N$, the label $ui$ corresponds to the $i$-th child of $u$, and we refer to $u$ as the parent of $ui$.
More generally, we write $u \preceq v$ if $v = uw$ for some $w \in \I$,
and say that $u$ is an ancestor of $v$, and $v$ is a descendant of $u$.  
If, in addition, $u \neq v$, we write $u \prec v$.
For $u = u_1 \ldots u_n \in \I$ and $k \in \N_0$, we denote by $u|_k$ the ancestor of $u$ in generation $k$
if $k \leq |u|$ and $u|_k = u$, otherwise.

Let $\bxi_u=(\xi_u^{i,j})_{i,j = 1,\ldots,p}$, $u \in \I \setminus \{\varnothing\}$
be a family of copies of $\bxi \eqqcolon \bxi_\varnothing$.
We write
\begin{align*}
\xi_u^{i} = \sum_{j=1}^p \xi_u^{i,j} \e_j.
\end{align*}
Let $(\Omega,\A,\Prob)$ be a probability space on which all $\bxi_u$, $u \in \I$ are defined and i.\,i.\,d.\ and
independent of $\tau(\varnothing)$, the random variable that gives the ancestor's type
and write $\Prob^i$ if the ancestor's type is $i \in \p$.
We denote the associated expected value operators by $\E$ and $\E^i$, respectively.

Generation~$0$ of the process is given by $\Gen_0 \coloneqq \{\varnothing\}$.  
The first generation is represented by the random set
\begin{equation*}
\Gen_1 \coloneqq \{k \in \N : X_k < \infty\} \subseteq \N,
\end{equation*}
where each $k \in \Gen_1$ is assigned a type $\tau(k) \in \p$.
Now suppose that generation $\Gen_n$ has been constructed and let $u \in \Gen_n$ be an individual of type $\tau(u) \in \p$.  
Define the point process
\begin{equation*}
\xi_u \coloneqq  \sum_{j=1}^p \xi_u^{\tau(u),j} = \sum_{k=1}^{N_u} \delta_{X_{u,k}},
\end{equation*}
where $N_u = \xi_u([0,\infty))$ is the total number of offspring of $u$ and
\begin{equation*}
X_{u,k} = \inf\{t \geq 0 : \xi_u([0,t]) \geq k\}, \quad k \in \N,
\end{equation*}
where the infimum of the empty set is set to be $\infty$.
Then the $(n+1)$-th generation is given by
\begin{equation*}
\Gen_{n+1} \coloneqq \{uk \in \I : u \in \Gen_n \text{ and } X_{u,k} < \infty\}.
\end{equation*}
The full population is described by the family tree
\begin{equation*}
\familytree \coloneqq \bigcup_{n \in \N_0} \Gen_n.
\end{equation*}
Each $u \in \familytree$ has a type $\tau(u) \in \p$ and a well-defined birth time $S(u) \in [0,\infty)$, given by
\begin{equation*}
S(u) = \sum_{k=1}^n X_{u|_{k-1}, u_k} \quad \text{for } u = u_1 \ldots u_n.
\end{equation*}
For $t \geq 0$ and $j \in \p$, define the counting process
\begin{equation*}
\cZ_t^j \coloneqq \sum_{u \in \familytree} \1_{\{j\}}(\tau(u)) \1_{[0,\infty)}(t - S(u)),
\end{equation*}
which counts the number of individuals of type~$j$ born up to and including time~$t$.
The corresponding vector process is
\begin{equation*}
\cZ_t \coloneqq (\cZ_t^1, \ldots, \cZ_t^p),
\end{equation*}
and $(\cZ_t)_{t \geq 0}$ is the multi-type general branching process.

\subsection{Assumptions}	\label{subsec:assumptions}

We define $\mu^{i,j} \defeq \E[\xi^{i,j}]$ as the intensity measure of $\xi^{i,j}$ for $i,j \in \p$,  
and let $\bmu \coloneqq (\mu^{i,j})_{i,j=1,\ldots,p} = \E[\bxi]$ denote the corresponding matrix.
We write $\cL\mu^{i,j}$ for the Laplace transform of $\mu^{i,j}$, i.e.,
\begin{align*}
\cL\mu^{i,j}:[0,\infty) \to [0,\infty],	\quad	\theta \mapsto \int e^{-\theta x} \, \mu^{i,j}(\dx),	\quad	i,j=1,\ldots,p
\end{align*}
and $\cL\bmu \defeq (\cL\mu^{i,j})_{i,j \in \p}$ for the corresponding matrix.
If $\cL\mu^{i,j}(\theta)<\infty$ for some $\theta \geq 0$, then, by monotonicity, $\cL\mu$ is finite on $[\theta,\infty)$.
We canonically extend $\cL\mu^{i,j}$ to $\Dom(\cL\mu^{i,j}) = \{z \in \C: \cL\mu^{i,j}(\Real(z))<\infty\}$.
The domain of the matrix-valued Laplace transform $\cL\bmu$ is then defined as the intersection of the domains of its components,
that is,
\begin{align*}
\Dom(\cL\bmu) = \bigcap_{i,j \in \p} \!\!\! \Dom(\cL\mu^{i,j})
= \bigg\{z \in \C: \int e^{-\Real(z) x} \, \mu^{i,j}(\dx) < \infty \text{ for all } i,j \in \p\bigg\}.
\end{align*}
Notice that $\cL\bmu(\theta)$ is a nonnegative matrix for every $\theta \in \Dom(\cL\bmu) \cap \R$
and hence possesses a maximal nonnegative eigenvalue
that equals the spectral radius $\rho_{\cL\bmu(\theta)}$ of $\cL\bmu(\theta)$.
Notice further, that $\rho_{\cL\bmu(\theta)}$ is decreasing and continuous in $\theta \in \Dom(\cL\bmu) \cap \R$,
see e.g.\ \cite[Lemma A.2]{Kolesko+al:2025+} for a reference.

The following assumptions hold throughout the paper.
\begin{enumerate}[{\bf{(A}1)}]
	\setcounter{enumi}{\value{assumptions}}
	\item
	The spectral radius $\rho_{\bmu(0)}$ of the matrix $\bmu(0) = (\mu^{i,j}(0))_{i,j\in\p}$
	satisfies $\rho_{\bmu(0)} < 1$.				\label{ass:subcritical instant offspring}
	\item
	There exists a \emph{Malthusian parameter}, i.e., there exists an $\alpha > 0$ 
	such that the spectral radius $\rho_{\cL\bmu(\alpha)}$ of $\cL\bmu(\alpha)$ equals $1$.
	\label{ass:Malthusian parameter}
\setcounter{assumptions}{\value{enumi}}
\end{enumerate}
We suppose that $\alpha > 0$ is maximal with the property $\rho_{\cL\bmu(\alpha)}=1$.
Notice that $\cL\bmu(\theta) \to \rho_{\bmu(0)}$ as $\theta \to \infty$ by the dominated convergence theorem,
hence the spectral radius $\rho_{\cL\bmu(\theta)}$ of $\cL\bmu(\theta)$ is strictly smaller than $1$ eventually,
proving that if there exists some $\alpha>0$ with $\rho_{\cL\bmu(\alpha)}$, then there is also a maximal one.

Further, notice that (A\ref{ass:subcritical instant offspring}) and (A\ref{ass:Malthusian parameter})
combined imply that $(\cZ_t)_{t \geq 0}$ does not explode in finite time,
see e.g.\ \cite[Proposition 2.1]{Kolesko+al:2025+} for a recent discussion.

Notice that (A\ref{ass:Malthusian parameter}) implies the supercriticality of the general branching process
in the irreducible (primitive) case and hence that the survival set
\begin{align}
\Surv \coloneqq \bigcap_{n\in\N} \{\#\Gen_n \geq 1\}
\end{align}
satisfies $\Prob^i(\Surv)> 0$ for all $i \in \p$.
To see this, first notice that the mean offspring matrix satisfies $\cL\bmu(0) \geq \cL\bmu(\alpha)$ component-wise.
Further, (A\ref{ass:Malthusian parameter}) implies that $\cL\mu^{i,j}(0) > \cL\mu^{i,j}(\alpha)$ for at least one pair $(i,j) \in \p^2$.
In fact, since $\cL\bmu(\alpha)$ has dominant eigenvalue $1$, it has at least one positive entry,
that is, there exists a pair $(i,j) \in \p^2$ with $\cL\mu^{i,j}(\alpha)>0$.
This means that $\mu^{i,j}(0,\infty) > 0$ and hence, by monotonicity, $\cL\mu^{i,j}(0) > \cL\mu^{i,j}(\alpha)$ as claimed.
From the Perron-Frobenius theorem \cite[Theorem 1.1(e)]{Seneta:1981} (after truncation of infinite entries if necessary),
we conclude that $\cL\bmu(0)$ has principle eigenvalue $\rho > 1$. This implies $\Prob^i(\Surv)>0$
by classical theory for multi-type Galton-Watson process, see e.g.\ \cite[Theorem II.7.1]{Harris:1963}.

However, we do not make the assumption that the branching process is irreducible (equivalently, that $\cL\bmu(\alpha)$ is a primitive,
i.e., has a positive matrix power).

\subsection{Construction of complex Nerman-type martingales}	\label{subsec:construction of Nerman-type martingales}

The multi-type counterpart of the single-type equation $\cL\mu(\lambda)=1$
is the following characteristic equation, see also \cite{Kolesko+al:2025+}:
\begin{align}	\label{eq:characteristic equation}
\det(\Ip - \cL\bmu(\lambda)) = 0,
\end{align}
where $\Ip$ denotes the $p \times p$ identity matrix and $\lambda \in \Dom(\cL\bmu)$.
Henceforth, let $\lambda \in \interior{\Dom(\cL\bmu)}$, the interior of $\Dom(\cL\bmu)$,  
be a solution of \eqref{eq:characteristic equation}.
Notice that (A\ref{ass:subcritical instant offspring}) and (A\ref{ass:Malthusian parameter})
imply that $\det(\Ip-\cL\bmu(z)) \not = 0$ for all $z$ in an open neighborhood of $\lambda$.
Indeed, $\det(\Ip-\cL\bmu(z))$ is a holomorphic function of $z \in \interior{\Dom(\cL\bmu)}$
and it is not constant as it vanishes at $\alpha$ and is non-zero eventually by the discussion following (A\ref{ass:Malthusian parameter}).
Thus, the set of zeros has no accumulation point in $\interior{\Dom(\cL\bmu)}$.

Whenever $z \in \Dom(\cL\bmu)$ is such that $\det(\Ip-\cL\bmu(z)) \not = 0$,
we can invert $\Ip-\cL\bmu(z)$ using the formula
\begin{align}
(\Ip-\cL\bmu (z))^{-1}=\frac{1}{\det(\Ip-\cL\bmu(z))}\; \adj(\Ip-\cL\bmu(z)),
\end{align}
where $\adj(A)$ denotes the adjoint of the matrix $A$.
Letting $z \to \lambda$, we see that each entry of $(\Ip-\cL\bmu(z))^{-1}$ has a pole with multiplicity at most
the multiplicity of the zero of $\det(\Ip-\cL\bmu(z))$ at $z=\lambda$.
We write $\klambda \in \N$ for the maximal order of the poles  
of the entries of $(\Ip - \cL\bmu(z))^{-1}$ at the point $\lambda$.

Further, we denote by $a^{i,j}_{\lambda,k}$ the coefficient of $(z - \lambda)^{-k}$
in the Laurent expansion of the $(i,j)$-th entry of $(\Ip - \cL\bmu(z))^{-1}$
around $\lambda$, for $k \in \N$.
Notice that $a^{i,j}_{\lambda,k}=0$ for $k>\klambda$.
We define $A_{\lambda,k} \defeq (a^{i,j}_{\lambda,k})_{i,j \in \p}$
as the matrix collecting all coefficients corresponding to the term $(z - \lambda)^{-k}$.
Consequently, the resolvent $(\Ip - \cL\bmu(z))^{-1}$ admits the representation
\begin{align}	\label{eq:(Ip-Lmu)^-1 Laurent}
(\Ip - \cL\bmu(z))^{-1} = \sum_{n=1}^{\klambda} A_{\lambda,n} (z - \lambda)^{-n} + h_1(z),
\end{align}
where $h_1(z)$ is a vector-valued function with holomorphic components.
Similarly, expanding $\Ip-\cL\bmu(z)$ around $\lambda$, we obtain the Taylor series
\begin{align}	\label{eq:(Ip-Lmu) Taylor}
\Ip-\cL\bmu(z) = \sum_{m=0}^{\klambda-1}\frac{(\Ip-\cL\bmu(\lambda))^{(m)}}{m!} (z-\lambda)^m+h_2(z)(z-\lambda)^{\klambda}
\end{align}
with $h_2$ denoting another vector-valued holomorphic function.
Multiplying \eqref{eq:(Ip-Lmu)^-1 Laurent} on the left by \eqref{eq:(Ip-Lmu) Taylor} yields
\begin{align}
\Ip&=(\Ip-\cL\bmu(z)) (\Ip-\cL\bmu(z))^{-1}	\notag \\
&=\biggl( \sum_{m=0}^{\klambda-1}\frac{(\Ip-\cL\bmu(\lambda))^{(m)}}{m!} (z-\lambda)^m+h_2(z)(z-\lambda)^{\klambda} \biggl) \biggl( \sum_{n=1}^{\klambda}A_{\lambda,n} (z-\lambda)^{-n}+h_1(z) \biggl).
\end{align}
By expanding the product and comparing the coefficients on the left- and right-hand sides, we conclude
\begin{align}	\label{eq:coefficient of (z-lambda)^(j-klambda)}
\sum_{m=0}^{\klambda-j}  \frac{(\Ip-\cL\bmu(\lambda))^{(m)}}{m!} A_{\lambda, m+j}=0 \quad  \text{for } j=1,\ldots,\klambda.
\end{align}
Hence, 
\begin{align} \label{eq:A_lambda identities}
A_{\lambda, j}=\sum_{m=0}^{\klambda-j}  \frac{\cL\bmu^{(m)}(\lambda)}{m!} A_{\lambda,m+j}  \quad  \text{for } j=1,\ldots,\klambda.
\end{align}
These identities can also be captured in matrix form, by writing
\begin{align}	\label{eq:matrix of matrices}
\begin{pmatrix}
A_{\lambda,1} \\
\vdots \\
A_{\lambda,\klambda}
\end{pmatrix}&=\begin{pmatrix}
\cL\bmu(\lambda) & \cL\bmu^{'}(\lambda) & \frac{\cL\bmu^{''}(\lambda)}{2!}   & \cdots & \frac{\cL\bmu^{(\klambda-1)}(\lambda)}{(\klambda-1)!}  \\
0 &\cL\bmu(\lambda) & \cL\bmu^{'}(\lambda) & \cdots & \vdots \\
0& 0  &\cL\bmu(\lambda) & \ddots &  \frac{\cL\bmu^{''}(\lambda)}{2!} \\
\vdots & \ddots    & \ddots & \ddots & \cL\bmu^{'}(\lambda)  \\
0  & \cdots & 0 &0  & \cL\bmu(\lambda)
\end{pmatrix} \begin{pmatrix}
A_{\lambda,1} \\
\vdots \\
A_{\lambda,\klambda}
\end{pmatrix}.
\end{align}
We rewrite the block matrix in \eqref{eq:matrix of matrices} in the following form:
\begin{align}	\label{eq:int exp mu}
\int \begin{pmatrix}
 e^{-\lambda s}  & -s e^{-\lambda s}  &   \frac{s^2}{2!} e^{-\lambda s}   & \cdots &   \frac{(-s)^{\klambda-1}}{(\klambda-1)!} e^{-\lambda s}   \\
0  & e^{-\lambda s}  &  -s e^{-\lambda s}  & \cdots &\vdots \\
0 &0   &  e^{-\lambda s} & \ddots &  \frac{s^2}{2!} e^{-\lambda s}   \\
\vdots & \ddots    & \ddots & \ddots &   -s e^{-\lambda s}    \\
0 &0 & 0  & 0 &  e^{-\lambda s} 
\end{pmatrix}
\otimes \bmu(\ds)
= \int \exp(\lambda,-s) \otimes \bmu(\ds)
\end{align}
where $\otimes$ denotes the Kronecker product, see Section~\ref{sec:Kronecker} in the appendix for its definition
and a summary of relevant properties,
and the $\klambda \times \klambda$ matrix $\exp(\lambda,x)$ for $x \in \R$
is given by
\begin{align}	\label{eq:exp matrix}
\exp(\lambda,x)
&= e^{\lambda x}
\begin{pmatrix}
1 & x  &   \frac{x^2}{2!}  & \cdots &   \frac{x^{\klambda-1}}{(\klambda-1)!}   \\
0  & 1  &  x   & \cdots &\vdots \\
0 &0   &  1 & \ddots &  \frac{x^2}{2!}   \\
\vdots & \ddots    & \ddots & \ddots &   x    \\
0 &0 & 0  & 0 &  1 
\end{pmatrix}.
\end{align}
Matrices of this form are particularly useful, as they simplify the notation and facilitate the handling of polynomial terms.  
Note that, for any $x, y \in \R$, we have
\begin{align}
\exp(\lambda,x) \cdot \exp(\lambda,y)=\exp(\lambda,x+y).
\end{align}
Writing $\f_1,\ldots,\f_{\klambda}$
for the canonical basis of $\R^{1 \times \klambda}$,
we define $\vA_{\lambda}=\sum_{l=1}^{\klambda}\f_l^\transp \otimes  A_{\lambda,l}$.
Then we can rewrite \eqref{eq:matrix of matrices} in the form
\begin{align} \label{eq: A_lambda=int exp mu A_lambda}
\vA_{\lambda}= \biggl( \int \exp(\lambda,-s) \otimes \bmu(\ds ) \biggl) \vA_{\lambda}
= \E\biggl[ \int \exp(\lambda,-s)  \otimes \bxi(\ds)\biggl]  \vA_{\lambda} .
\end{align}
Returning to \eqref{eq:A_lambda identities}, we for any $u \in \I$, from the basic calculation rules for the Kronecker product we obtain
\begin{align*}
(\Ik\otimes\e_{\tau(u)})\vA_{\lambda} &= (\Ik\otimes\e_{\tau(u)})  \E\biggl[ \int \exp(\lambda,-s)  \otimes \bxi(\ds)\biggl]  \vA_{\lambda}  \\
&=  \E\biggl[ \int \exp(\lambda,-s)  \otimes \bxi_u^{\tau(u)}(\ds) \, \Big|\, \tau(u) \biggl]  \vA_{\lambda}.
\end{align*}
Now define the i.i.d., integrable random $\klambda p\times p$ matrices
\begin{align*}
Z_u(\lambda)\defeq \biggl( \int \exp(\lambda,-s)  \otimes \bxi_u(\ds) \biggl)  \vA_{\lambda},	\quad	u \in \I,
\end{align*}
$Z(\lambda) \defeq Z_\varnothing(\lambda)$,
to establish a more compact formulation of \eqref{eq: A_lambda=int exp mu A_lambda},
namely, $ \E[Z(\lambda)] = \vA_{\lambda}$.
Additionally,
\begin{align}
Z^{\tau(u)}_u(\lambda)=(\Ik\otimes \e_{\tau(u)})Z_u(\lambda)= \int (\exp(\lambda,-s)  \otimes \bxi_u^{\tau(u)}(\ds) )\vA_{\lambda}
\end{align}
so that $\E[Z^{\tau(u)}_u(\lambda) | \tau(u)]=(\Ik\otimes \e_{\tau(u)})\vA_{\lambda}$.
Define the random ($\klambda p \times p$)-matrix $Y_u \defeq Z_u(\lambda)-\vA_{\lambda}$, $u \in \I$.
As usual, we let $Y \defeq Y_\varnothing$.
Here, a basic observation is that $\E[Y]=0$ by \eqref{eq: A_lambda=int exp mu A_lambda}.

We now define the Nerman-type complex martingales in the multi-type setting.
To this end, we first define the $(kp\times p)$-matrix-valued characteristic $\phil$
which is central for our analysis, as it enables us to introduce the martingales
as multi-type general branching process counted with random characteristic $\phil$.
For $u \in \I$ and $t\in\R$, we define
\begin{align}	\label{eq:philambda}
\philu(t)\defeq \1_{[0,\infty)}(t)   \biggl(\int \1_{(t,\infty)}(s) \exp(\lambda,t-s)  \otimes \bxi_u(\ds) \biggl) \vA_{\lambda}
\end{align}
and $\phil \defeq \phi_{\lambda,\varnothing}$.
Notice that $\philu(t)$ is integrable with respect to $\Prob$ for every $t \geq 0$.
Now define
\begin{align} \label{eq:martingale}
W_t(\lambda) \defeq \exp(\lambda, -t) \cZ^{\phil}_t
\defeq \exp(\lambda, -t) \sum_{u \in \I} (\Ik \otimes \e_{\tau(u)}) \philu(t-S(u)), \quad t \geq 0.
\end{align}
We define $\F_t\defeq\sigma(\{ A\cap \{ S(u)\leq t\}:u\in\I,\ A\in\Gen_u \})$ for $t \geq 0$,
with $\A_u\defeq\sigma(\bxi_v: v \preceq u)$, $u \in \I$.
It is straightforward to check that $(\F_t)_{t \geq 0}$ is a filtration.

\begin{theorem}	\label{Thm: martingale}
Suppose that (A\ref{ass:subcritical instant offspring}) and (A\ref{ass:Malthusian parameter}) hold,
and let $\lambda \in \interior{\Dom(\cL\bmu)}$ satisfy $\det(\Ip-\cL\bmu(\lambda))=0$.
Then, $(W_t(\lambda))_{t \geq 0}$ is a martingale with respect to the filtration $(\F_t)_{t \geq 0}$ with 
\begin{align}	\label{eq:E^tau(u)W_t(lambda)}
\E^{\tau(\varnothing)}[W_t(\lambda)] =W_0(\lambda)=(\Ik \otimes \e_{\tau(\varnothing)}) \vA_\lambda.
\end{align}
Further, $W_t(\lambda)$ can be rewritten in the form
\begin{align}
W_t(\lambda )= \sum_{u\in \cC_t}    (\exp(\lambda, -S(u) ) \otimes \e_{\tau(u)}) \vA_\lambda
\label{eq: martingale2}
\end{align}
where $\cC_t \defeq \{uj \in \familytree : S(u) \leq t < S(uj)\}$ denotes the coming generation at time $t$,  
that is, the set of all individuals born strictly after time $t$, whose parents were born at or before time $t$.
\end{theorem}

\begin{remark}
	Let $\alpha$ be such that the matrix
	$\cL\bmu(\alpha)$ is primitive and has spectral radius~$1$. As we described in the Introduction, the classical Nerman's martingale in such a case is a multiple of 
	\[
	V_t\defeq\sum_{u\in\cC_t} v_{\tau(u)}e^{-\alpha S(u)}  .
	\]
	A natural question is how $V_t$ is related to the matrix‑valued martingale
	$W_t(\alpha)$ introduced in our Theorem~\ref{Thm: martingale}.  Because every coordinate of $W_t(\alpha)$ is itself a martingale,
	one might wonder whether still other real valued martingales can be obtained, for instance by inserting terms such as $S(u)^j e^{-\alpha S(u)}$.
	
	We claim  that in this setting the order of singularity 
	$\mathtt{k}_\alpha$ equals~$1$ and that
	\[
	A_{\alpha,1} = a\,\pi_{\mathbf v},
	\]
	where $a\in\mathbb{R}$ and
	$\pi_{\mathbf v}=\mathbf v\mathbf w$ is the orthogonal projection onto the
	one‑dimensional eigenspace spanned by~$\mathbf v$; here $\mathbf w$ is the unique
	left (row) eigenvector of $\cL\bmu(\alpha)$ with
	$\mathbf w\mathbf v = 1$.  Consequently,
	\begin{align*}
		W_t(\alpha)
		=\sum_{u\in\cC_t} e^{-\alpha S(u)}\,
		\e_{\tau(u)}A_{\alpha,1}
		= a\sum_{u\in \cC_t} e^{-\alpha S(u)}\,\mathbf v_{\tau(u)}\,\mathbf w
		= a\,V_t\,\mathbf w,
	\end{align*}
	so every coordinate of $W_t(\alpha)$ is merely a multiple of the martingale $V_t$.
	
	It remains to justify the claim above.  For
	$s\in\Dom(\cL\bmu)$ define
	$\rho(s)=\rho (\cL\bmu(s))$, the spectral radius of
	$\cL\bmu(s)$.  Perron–Frobenius implies that $\rho(s)$ is the
	dominant eigenvalue, all others having strictly smaller modulus.  Because the eigenvalues of a matrix depend continuously on its coefficients (see, e.g.,
	\cite[Theorem~3.1.2]{Ortega:1972}), we may write
	\begin{align*}
	\det (I-\cL\bmu(s))
	\;=\; (1-\rho(s) )\,h(s),
	\end{align*}
	where $h$ is continuous and bounded away from~$0$ in a neighbourhood of~$\alpha$.
	It is known that the function $\rho$ is strictly decreasing and convex (see \cite[Proposition~A.1. and Lemma~A.2.]{Kolesko+al:2025+}). This implies that $\alpha$ is a simple root $\rho(s)=1$. Therefore, the function $\det\bigl(I - \mathcal{L}\boldsymbol{\mu}(s)\bigr)$ has a simple zero at $s = \alpha$, which in turn means that the matrix-valued function $(I - \mathcal{L}\boldsymbol{\mu}(s))^{-1}$ has a simple pole at~$\alpha$. Consequently, the multiplicity $\mathtt{k}_\alpha$ equals~$1$.

	Finally, from~\eqref{eq:matrix of matrices} we have
	$A_{\alpha,1}= \mathcal{L}\boldsymbol{\mu}(\alpha)A_{\alpha,1}$. Similar argument yields $A_{\alpha,1}=A_{\alpha,1} \mathcal{L}\boldsymbol{\mu}(\alpha)$.
	Iterating these identities and letting the number of iterations tend to infinity gives
	\[
	A_{\alpha,1}=\pi_{\mathbf v}A_{\alpha,1},
	\qquad
	A_{\alpha,1}=A_{\alpha,1}\pi_{\mathbf v},
	\]
	so $\operatorname{im}(A_{\alpha,1})=\mathbb{R}\mathbf v$ and
	$\ker (A_{\alpha,1})=\{\mathbf v\}^\perp$.  Thus $A_{\alpha,1}$ is indeed a scalar
	multiple of the projection~$\pi_{\mathbf v}$.
\end{remark}

\subsection{$L^q$-convergence for $1 < q \leq 2$}	\label{subsec:L^q convergence}

Consider the situation of Theorem \ref{Thm: martingale}.
In this section, we provide sufficient conditions for the convergence of $(W_t(\lambda))_{t \geq 0}$ in $L^q$ for $1 < q \leq 2$.  
While our primary interest---motivated by a central limit theorem for the multi-type Crump-Mode-Jagers process---lies in the case $q = 2$,
we also include the cases $1 < q < 2$, as they require little additional effort.
We consider the martingale in the form
\begin{align*}
W_t(\lambda)
=\1_{[0,\infty)}(t) (\Ik \otimes \e_{\tau(\varnothing)}) \vA_\lambda+ \sum_{u\in\I} \1_{\{S(u)\leq t\}} (\exp(\lambda,-S(u) ) \otimes \e_{\tau(u)}) Y_u.
\end{align*}
Naturally, for $(W_t(\lambda))_{t \geq 0}$ to be an $L^q$-martingale,  
we require that $\E[\|Y_u\|^q] < \infty$, where $\|\cdot\|$ denotes the operator matrix norm.  
Equivalently, this condition holds with the Hilbert--Schmidt norm, i.e., $\E[\|Y_u\|_\HS^q] < \infty$,  
where $\|\cdot\|_\HS$ denotes the Hilbert--Schmidt norm.  
As we shall see, the following assumption is sufficient to guarantee this:
\begin{enumerate}[{\bf{(A}1)}]
	\setcounter{enumi}{\value{assumptions}}
	\item	$\displaystyle	\E\biggl[ \biggl\lVert \int\Big(1+ s^{\klambda-1}\Big) e^{-\theta s} \, \bxi(\ds) \biggl\rVert^q \biggl]<\infty$
	\label{ass:L^q norm}
\setcounter{assumptions}{\value{enumi}}
\end{enumerate}
where $\theta = \Real(\lambda)$.

Our main result regarding convergence in $L^q$ is the following:

\begin{theorem}	\label{Thm:martingale convergence L^q}
Suppose that (A\ref{ass:subcritical instant offspring}) and (A\ref{ass:Malthusian parameter}) hold,
and let $\lambda \in \interior{\Dom(\cL\bmu)}$ satisfy $\det(\Ip-\cL\bmu(\lambda))=0$.
Further suppose that, for some $1 < q \leq 2$, (A\ref{ass:L^q norm}) holds.
If $\Real(\lambda)>\frac{\alpha}{q}$, then $(W_t(\lambda))_{t \geq 0}$ is an $L^q$-bounded martingale.
\end{theorem}

\begin{remark}
	The condition $\Real(\lambda)>\frac{\alpha}{q}$ is sharp only in the primitive case, i.e., when $\cL\bmu(\alpha)$
	has a matrix power with positive entries only. In the reducible case, the condition may be relaxed in some cases.
	We refrain from providing more details.
\end{remark}

\section{Examples}	\label{sec:Examples}

Before turning to the proofs of our main results, we first present a few examples.
The first example illustrates what can happen in the non-primitive case,
namely, if the process is started with a type that is subcritical, the corresponding martingale may be trivial.

\begin{example}
Consider a multitype general branching process with two types, labeled $1$ and $2$.  
Individuals of type $1$ produce both type-$1$ and type-$2$ offspring at constant rate $1$,  
whereas individuals of type $2$ do not produce type-$1$ offspring.
Instead, each type-$2$ individual either has a single child at a random (exponential) time with probability $1/2$,
or no offspring at all, also with probability $1/2$.  
All reproduction events are assumed to be independent.

More formally, we assume that $\xi^{1,1}$, $\xi^{1,2}$, $B$, and $\zeta$ are independent.  
The processes $\xi^{1,1}$ and $\xi^{1,2}$ are homogeneous Poisson point processes on $[0, \infty)$ with rate $1$.  
Type $2$-individuals do not produce type $1$ offspring, i.e., $\xi^{2,1} = 0$, and produce at most one type-$2$ offspring:
specifically, $\xi^{2,2} = B \delta_{\zeta}$,  
where $B \in \{0,1\}$ is a Bernoulli random variable with $\Prob(B = 1) = \Prob(B = 0) = \frac{1}{2}$,  
and $\zeta$ is an exponential random variable with unit mean.

Note that if the population starts with an individual of the second type,
the entire population eventually becomes extinct, as the corresponding Galton--Watson process is subcritical.
The associated Laplace transform $\cL\bmu(z)$ takes the form:
\begin{align*}
\cL\bmu(z) =
\begin{pmatrix}
\vspace{3pt} \tfrac{1}{z} & \tfrac{1}{z} \\
0 & \tfrac{1}{2(1+z)}
\end{pmatrix}.
\end{align*}
One can easily verify that $\alpha = 1$ and
\begin{align*}
(z-1)(\Ip - \cL\bmu(z))^{-1} \to
\begin{pmatrix}
1 & \tfrac{4}{3} \\
0 & 0
\end{pmatrix},
\end{align*}
as $z \to 1$. The martingale $W_t(1)$ corresponding to the parameter $\alpha = 1$ is given by
\begin{align*}
W_t(1) = \sum_{u \in \cC_t} e^{-S(u)} \e_{\tau(u)}
\begin{pmatrix}
1 & \tfrac{4}{3} \\
0 & 0
\end{pmatrix}
= \sum_{\substack{u \in \cC_t:\\ \e_{\tau(u)} = 1}} e^{-S(u)}
\begin{pmatrix}
1 & \tfrac{4}{3}
\end{pmatrix},
\end{align*}
and $W_0(1) = \1_{\{\tau(\varnothing) = 1\}}
\begin{pmatrix}
1 & \tfrac{4}{3}
\end{pmatrix}$.
In particular, when the branching process starts with an individual of type 2,
the martingale $W_t(1) = (0\, 0)$ for all $t \geq 0$.
\end{example}

\begin{example}
	Consider the process introduced in Example~4.5 of~\cite{Kolesko+al:2025+}. It is a two-type process in which particles of type 1 give birth to type-1 particles according to a Poisson point process with intensity $\alpha>0$, and particles of type 2 independently produce type 2 offspring, also according to a Poisson point process with intensity $\alpha$. Additionally, each type 1 particle immediately produces exactly one type 2 particle at its time of birth.
	
	In this case, the multiplicity is $\mathtt k_\alpha=2$, and the corresponding matrices associated with the singularities are given by
	\begin{align*}
		A_{\alpha,2}=
		\begin{pmatrix}
			0 & \alpha^2\\
			0 & 0
		\end{pmatrix}
		\quad \text{and} \quad
		A_{\alpha,1}=
		\begin{pmatrix}
			\alpha & 2\alpha\\
			0 & \alpha
		\end{pmatrix}.
	\end{align*}
	This allows us to explicitly express the martingale $W_t(\alpha)$ as follows:
	\begin{align*}
		W_t(\alpha)= &\sum_{u\in \cC_t} e^{-\alpha S(u)}  
		\left( \begin{pmatrix}
			1 & -S(u)\\
			0 & 1
		\end{pmatrix}
		\otimes \e_{\tau(u)}\right) (\f_1^\transp \otimes  A_{\alpha,1} + \f_2^\transp \otimes  A_{\alpha,2})\\
		= &\sum_{u\in \cC_t} e^{-\alpha S(u)}  
		\bigg( \begin{pmatrix}
			1 \\
			0 
		\end{pmatrix}
		\otimes 
		\Big(\begin{pmatrix}
			\alpha & 2\alpha
		\end{pmatrix}\1_{\{\tau(u) = 1\}}
		+
		\begin{pmatrix}
			0 & \alpha
		\end{pmatrix}\1_{\{\tau(u) = 2\}}\Big)		\\
		&\phantom{\sum_{u\in \cC_t} e^{-\alpha S(u)}  
			\bigg( }+
		\begin{pmatrix}
			-S(u) \\
			1
		\end{pmatrix}
		\otimes 
		\begin{pmatrix}
			0 & \alpha^2
		\end{pmatrix}\1_{\{\tau(u) = 1\}}
		\bigg)\\
		=&\sum_{u\in \cC_t} e^{-\alpha S(u)}
		\begin{pmatrix}
			\alpha \1_{\{\tau(u) = 1\}} & \big(2\alpha-\alpha^2 S(u)\big)\1_{\{\tau(u) = 1\}} + \alpha\1_{\{\tau(u) = 2\}}\\
			0 & \alpha^2\1_{\{\tau(u) = 1\}}
		\end{pmatrix}.
	\end{align*}
	Thus, up to multiplicative constants, there are exactly two distinct martingales:
	\begin{align*}
		\sum_{u\in \cC_t} e^{-\alpha S(u)}\1_{\{\tau(u) = 1\}} \quad\text{and}\quad
		\sum_{u\in \cC_t} e^{-\alpha S(u)} \big[\big(2-\alpha S(u)\big)\1_{\{\tau(u) = 1\}} + \1_{\{\tau(u) = 2\}}\big].
	\end{align*}
	Moreover, under the measure $\Prob^1$, both the martingales are non-degenerate.
\end{example}

\section{Proofs of the main results}	\label{sec:proofs}

\begin{proof}[Proof of Theorem \ref{Thm: martingale}]
First, we consider $\cZ^{\phil}_t$ for $t \geq 0$. We write
\begin{align*}
\cZ^{\phil}_t &= \sum_{u\in\I} (\Ik\otimes \e_{\tau(u)}) \philu(t-S(u)) \\
&= \sum_{u\in\I} \1_{\{S(u)\leq t\}}  \biggl(\int_{(t-S(u),\infty)} \exp(\lambda,t-S(u)-s)  \otimes \bxi_u^{\tau(u)}(\ds) \biggl)\vA_\lambda.
\end{align*}
Notice that by (A\ref{ass:subcritical instant offspring}) and (A\ref{ass:Malthusian parameter}),
for any $t \geq 0$,
with probability $1$, we have $S(u)> t$ for all but finitely many $u \in \I$.
Hence the sum above is almost surely a finite sum of integrable random variables and thus finite almost surely,
showing that $\cZ^{\phil}_t$ and hence also $W_t(\lambda)$ is well-defined.
Using $\exp(\lambda,-t)=\exp(\lambda,-t)\otimes 1$, we obtain
\begin{align} \label{eq:W_t starting point}
W_t(\lambda) = \exp(\lambda,-t)\cZ^{\phil}_t
&= \sum_{u\in\I} \1_{\{S(u)\leq t\}}  \biggl(\int_{(t-S(u),\infty)} \exp(\lambda,-S(u)-s)  \otimes \bxi_u^{\tau(u)}(\ds) \biggl)\vA_\lambda \notag \\
&= \sum_{u\in\I} \sum_{j=1}^{N_u}\1_{\{S(u)\leq t<S(uj)\}} (\exp(\lambda,-S(uj))  \otimes \e_{\tau(uj)}) \vA_\lambda .
\end{align}
Substituting $\1_{\{S(u)\leq t<S(uj)\}}=\1_{\cC_t}(uj)$ in \eqref{eq:W_t starting point} gives \eqref{eq: martingale2}.
We further deduce
\begin{align*}
W_t(\lambda)
&= \sum_{u\in\I} \sum_{j=1}^{N_u}\1_{\{S(u)\leq t<S(uj)\}} (\exp(\lambda,-S(uj))  \otimes \e_{\tau(uj)}) \vA_\lambda \\
&= \sum_{u\in\I} \sum_{j=1}^{N_u}(\1_{\{S(u)\leq t\}}-\1_{\{S(uj)\leq t\}}) (\exp(\lambda,-S(uj))  \otimes \e_{\tau(uj)}) \vA_\lambda	\\
&= \sum_{u\in\I}   \1_{\{S(u)\leq t\}} (\exp(\lambda,-S(u))\otimes 1) \biggl( \sum_{j=1}^{N_u}  \exp(\lambda,-X_{u,j}) \otimes \e_{\tau(uj)}  \biggl)   \vA_\lambda   \\
&\hphantom{=}-\sum_{|u|\geq 1} \1_{\{S(u)\leq t\}}  (\exp(\lambda,-S(u)) \otimes \e_{\tau(u)})  \vA_\lambda.
\end{align*}
Note that the decomposition of the sum is justified, since both sums almost surely contain only finitely many nonzero terms.
We proceed by writing
\begin{align}
W_t(\lambda)
&= \sum_{u\in\I} \1_{\{S(u)\leq t\}} (\exp(\lambda,-S(u))\otimes 1) \biggl(\sum_{j=1}^{N_u} \exp(\lambda,-X_{u,j}) \otimes \e_{\tau(uj)}  \biggl)   \vA_\lambda   \notag	\\
&\hphantom{=}-\sum_{|u|\geq 1} \1_{\{S(u)\leq t\}}  (\exp(\lambda,-S(u)) \otimes \e_{\tau(u)})  \vA_\lambda	\notag \\
&=\1_{[0,\infty)}(t)     (\Ik \otimes \e_{\tau(\varnothing)})\vA_\lambda+ \sum_{u\in\I}    \1_{\{S(u)\leq t\}} (\exp(\lambda,-S(u))\otimes \e_{\tau(u)}) Y_u.	\label{eq:W_tlambda Y_u representation}
\end{align}
At this point, it can be readily seen that $W_t(\lambda)$ is $\F_t$-measurable.  
Indeed, this follows from the fact that each summand in the preceding expression is $\F_t$-measurable,  
which, in turn, can be verified directly from the definition of $\F_t$ and basic measurability arguments.
The next item on the agenda is the proof of the integrability of $W_t(\lambda)$.
We estimate as follows, using the matrix operator norm $\|\cdot\|$,
\begin{align*}
\E\bigg[\bigg\|\sum_{u\in\I} \1_{\{S(u)\leq t\}} (\exp(\lambda,-S(u))\otimes \e_{\tau(u)}) Y_u\bigg\|\bigg]
&\leq \sum_{u\in\I} \E\big[\1_{\{S(u)\leq t\}} \|\exp(\lambda,-S(u))\otimes \e_{\tau(u)}\| \| Y_u\|\big]	\\
&\leq C \E\bigg[\sum_{u\in\I} \big[\1_{\{S(u)\leq t\}}\bigg]
\end{align*}
for some finite constant $C>0$. It is well known in the theory of multi-type branching processes,
see e.g.\ \cite[Proposition 2.1]{Kolesko+al:2025+} that
under the assumptions (A\ref{ass:subcritical instant offspring}) and (A\ref{ass:Malthusian parameter})
the expected number of births up to and including time $t$ is finite for every $t \geq 0$.

The final property to be verified is the martingale property.  
To this end, let us fix $0 \leq s < t$. Then,
\begin{align*}
W_t(\lambda)-W_s(\lambda) &= \sum_{u\in\I}  \1_{\{s<S(u)\leq t\}} (\exp(\lambda,-S(u))\otimes \e_{\tau(u)}) Y_u.
\end{align*}
Here, it suffices to show that
\begin{align*}
\E\big[  \1_{\{s<S(u)\leq t\}}   (\exp(\lambda,-S(u))\otimes\e_{\tau(u)}) Y_u \, | \, \F_s \big]=0 \quad \text{a.\,s. }
\end{align*}
To prove this, we first claim that, for any $v \in \I$ and $A\in \A_v$,
\begin{align*}
\E [ \1_{A\cap\{S(v)\leq s \}}  \1_{\{s<S(u)\leq t\}}   (\exp(\lambda,-S(u) )\otimes\e_{\tau(u)}) Y_u]=0.
\end{align*}
Indeed, if $u \preceq v$, then $S(u) \leq S(v)$, and thus $\1_{\{S(v) \leq s\}} \, \1_{\{s < S(u) \leq t\}} = 0$.  
On the other hand, if $u \not\preceq v$, then $\sigma(\A_v \cap \A_{u|_{|u|-1}})$ is independent of $\bxi_u$, and in particular of $Y_u$.  
Consequently, the expectation again vanishes, since $\E[Y_u] = 0$.
The argument carries over if we take a finite intersection of sets of the form $A_u \cap \{S(u)\leq s\}$ for $A_u\in\A_u$
and different $u\in\I$. Dynkin's $\pi$-$\lambda$ theorem then implies the martingale property. In particular, \eqref{eq:E^tau(u)W_t(lambda)} holds.
\end{proof}

\begin{proof}[Proof of Theorem \ref{Thm:martingale convergence L^q}]
Throughout the proof, we write $\theta \defeq \Real(\lambda)$.
We first show that (A\ref{ass:L^q norm}) implies $\E[\|Y\|^q]<\infty$.
Using the submultiplicative property of matrix norms, we obtain,
with $\|\cdot\|_\HS$ denoting the Hilbert-Schmidt norm,
\begin{align*}
\E[\lVert Y\rVert^q_\HS]&= \E\biggl[ \biggl\lVert \biggl( \int \exp(\lambda,-s)  \otimes \bxi(\ds)-(\Ik\!\otimes\Ip) \biggl)  \vA_{\lambda}  \biggl\rVert^q_\HS \biggl] \\
&\leq \lVert  \vA_{\lambda}\rVert^q_\HS  \E\biggl[ \biggl\lVert \biggl( \int \exp(\lambda,-s)  \otimes \bxi(\ds)-(\Ik\!\otimes\Ip) \biggl)  \biggl\rVert^q_\HS  \biggl] .
\end{align*}
Since for any $p\times p$ matrix $A$, it holds that $\lVert A-\Ip \rVert^q_\HS\leq (\lVert A \rVert_\HS+\lVert \Ip \rVert_\HS)^q \leq c_q(\lVert A \rVert^q_\HS+1)$ for some finite constant $c_q>0$, we conclude
\begin{align*}
\E[\lVert Y\rVert^q_\HS]
&\leq
c_q \lVert  \vA_{\lambda}\rVert^q_\HS
\E\biggl[ \biggl\lVert \biggl( \int \exp(\lambda,-s)  \otimes \bxi(\ds)-(\Ik\!\otimes\Ip) \biggl)  \biggl\rVert^q_\HS  \biggl] \\
&\leq c_q \lVert  \vA_{\lambda}\rVert^q_\HS \biggl(1+  \E\biggl[ \biggl\lVert \biggl( \int \exp(\lambda,-s)  \otimes \bxi(\ds) \biggl)  \biggl\rVert^q_\HS  \biggl] \biggl) \\
&\leq C c_q \lVert  \vA_{\lambda}\rVert^q_\HS \biggl(1+ \sum_{l=0}^{\klambda -1} \sum_{i,j=1}^p \E\biggl[ \biggl( \int e^{-\Real(\lambda) s} s^{l}  \, \xi^{i,j}(\ds) \biggl)^q \biggl] \biggl).
\end{align*}
for some constant $C > 0$, where in the final step we used the fact that the function $x \mapsto x^{q/2}$ is subadditive on $[0, \infty)$.  
The finiteness of the last expression follows from assumption (A\ref{ass:L^q norm}).

We now turn to the $L^q$-boundedness, i.e., $\sup_{t \geq 0} \E[\|W_t(\lambda)\|^q] < \infty$,
which, in view of \eqref{eq:W_tlambda Y_u representation}, is equivalent to
\begin{equation}	\label{eq:W_t(lambda) L^q-bounded}
\sup_{t \geq 0} \E\bigg[\bigg\|\sum_{u\in\I} \1_{\{S(u)\leq t\}} (\exp(\lambda,-S(u))\otimes \e_{\tau(u)}) Y_u\bigg\|^q \bigg] < \infty.
\end{equation}
Let $x\in\R^{1\times k}$ and $y\in \R^{p\times 1}$ be fixed vectors. Then, $(x\otimes 1)$ and $(1\otimes y)$ are vectors in the space $\R^{1\times k}\otimes\R$ and $\R\otimes \R^{p\times 1}$, respectively.
We assume that $|x|, |y| \leq 1$, where $|\cdot|$ denotes the Euclidean norm, here and throughout.
For \eqref{eq:W_t(lambda) L^q-bounded} to hold, it is sufficient to show that
\begin{align*}
(x \otimes 1) \sum_{u \in \I} \1_{\{S(u)\leq t\}} (\exp(\lambda,-S(u)) \otimes \e_{\tau(u)}) Y_u (1\otimes y)	\\
= \sum_{u \in \I} \1_{\{S(u)\leq t\}} (x \exp(\lambda,-S(u)) \otimes \e_{\tau(u)}) Y_u (1\otimes y)
\end{align*}
is an $L^q$-bounded martingale. Define 
\begin{align*}
M_n(\lambda)\defeq \sum_{|u|\leq n} \1_{\{S(u)\leq t\}} (x \exp(\lambda,-S(u))\otimes \e_{\tau(u)}) Y_u (1\otimes y)
\end{align*}
and notice that its $n$th increment, namely,
\begin{align*}
M_n(\lambda)-M_{n-1}(\lambda)=\sum_{|u|=n}   \1_{\{S(u)\leq t\}} (x \exp(\lambda,-S(u))\otimes \e_{\tau(u)}) Y_u (1\otimes y)
\end{align*}
is, conditionally on $\sigma(\tau(\varnothing), \bxi_u : |u| < n)$, a weighted sum of independent, centered random variables, for $n \in \N_0$.  
Therefore, we may apply the von Bahr--Esseen inequality (see \cite[Theorem 2]{Bahr+Esseen:1965} for the original version
and \cite[Lemma A.1]{Iksanov+Kolesko+Meiners:2020} for a version for complex-valued random variables) twice: once directly,
and once conditionally on $\sigma(\tau(\varnothing), \bxi_u : |u| < n)$, to obtain
\begin{align*}
\E\biggl[&\biggl|\sum_{u \in \I} \1_{\{S(u)\leq t\}} (x \exp(\lambda,-S(u))\otimes \e_{\tau(u)}) Y_u (1\otimes y)\biggl|^q\biggl] \\
&\leq 4 \sum_{n=0}^\infty \E\biggl[\biggl| \sum_{|u|=n} \1_{\{S(u)\leq t\}}(x \exp(\lambda,-S(u))\otimes \e_{\tau(u)}) Y_u (1\otimes y)\biggl|^q\biggl]	\\
&\leq 4 \sum_{n=0}^\infty \E\bigg[\E\biggl[\biggl| \sum_{|u|=n}   \1_{\{S(u)\leq t\}}(x \exp(\lambda,-S(u))\otimes \e_{\tau(u)}) Y_u (1\otimes y)\biggl|^q \, \bigg|\,\tau(\varnothing), \bxi_u : |u| < n \biggl]\bigg] \\
&\leq 16 \, \E\biggl[  \sum_{u \in \I}  \1_{\{S(u)\leq t\}} |(x \exp(\lambda,-S(u))\otimes \e_{\tau(u)}) Y_u (1\otimes y) |^q  \biggl].
\end{align*}
Then, using $|x|,|y|\leq 1$ and the independence of $\1_{\{S(u)\leq t\}} ( \exp(\lambda,-S(u) )\otimes \e_{\tau(u)})$ and $Y_u$,
\begin{align*}
\sum_{u\in\I} \E[ |  \1_{\{S(u)\leq t\}}(x \exp(\lambda,-S(u))\otimes \e_{\tau(u)}) Y_u (1\otimes y) |^q]
\leq \sum_{n=0}^\infty \sum_{|u|=n} \E\biggl[\lVert \exp(\lambda,-S(u)) \rVert^q\biggl] \E[ \lVert Y_u  \rVert^q].
\end{align*}
Using $\E[\|Y\|^q] \leq \E[\|Y\|_\HS^q]<\infty$ and the condition $\theta > \alpha/q$,
we find $\delta > 0$ so small that $q(\theta - \delta)>\alpha$ and some finite constant $C_\delta>0$ such that
\begin{align*}
\sum_{n=0}^\infty \sum_{|u|=n}   \E\biggl[\lVert \exp(\lambda,-S(u)) \rVert^q\biggl] \E[ \lVert Y_u  \rVert^q]
&\leq C_\delta \E[\|Y\|^q] \sum_{n=0}^\infty \E\biggl[ \sum_{|u|=n} e^{-q (\theta-\delta) S(u)}\biggl] \\
&\leq C_\delta \E[\|Y\|^q] \E[\e_{\tau(\varnothing)}] \sum_{n=0}^\infty  \cL\bmu(q(\theta-\delta))^n \mathbf{1}^\transp \\
&\leq C_\delta \E[\|Y\|^q] \E[\e_{\tau(\varnothing)}] (\Ip-\cL\bmu(q(\theta-\delta))^{-1} \mathbf{1}^\transp < \infty
\end{align*}
where $\mathbf{1}\in\R^{1\times p}$ denotes the row vector with all entries equal to one.
Further, notice that we have used that $q(\theta - \delta)>\alpha$,
hence the spectral radius of $\cL\bmu(q(\theta-\delta))$ is strictly smaller than $1$ and the infinite series of matrices converges.
The proof is complete.
\end{proof}

\section*{Acknowledgement}
Matthias Meiners and Ivana Tomic were supported by DFG grant ME3625/4-1.
Part of this work was conducted during visits of Matthias Meiners and Ivana Tomic to the University of Wroc\l aw,
for which they express their gratitude for the warm hospitality.

\begin{appendix}
\section{The Kronecker product}	\label{sec:Kronecker}
Consider a $m\times n$ matrix $A = (a_{ij})_{i=1,\ldots,m,j=1,\ldots,n}$ and a $k\times l$ matrix $B$.
Following \cite[Chapter 2.2]{Graham:1981}, the Kronecker product of the matrices $A$ and $B$, denoted by $A \otimes B$,
is defined as the $mk\times nl$ matrix 
\begin{align*}
A\otimes B
\defeq
\begin{pmatrix}
a_{11} B & \cdots & a_{1n} B  \\
\vdots & \ddots  & \vdots \\
a_{m1} B  & \cdots &  a_{mn} B
\end{pmatrix}.
\end{align*}
In other words, the Kronecker product increases the dimensions of both matrices
by replacing each entry of $A$ with a scaled copy of the entire matrix $B$, resulting in a block matrix.

\section{Exponential Matrices}	\label{sec:Matrix notation}

For $\lambda \in \interior{\Dom(\cL\bmu)}$ satisfying \eqref{eq:det(Lbmu-Ip)=0},
we use a multiplicative family of matrices $\exp(\lambda,x)$, $x\in \R$
similar to that used in \cite[Equation (4.10)]{Iksanov+al:2024},
namely,
we define the following upper triangular $\klambda \times \klambda$ matrix
\begin{align}
\exp(\lambda, x)
&= (\exp_{ij}(\lambda, x))_{i,j=1,\ldots,\klambda}
\defeq e^{\lambda x}   \begin{pmatrix}
1 & x & \frac{x^2}{2!}  & \cdots & \frac{x^{\klambda-1} }{(\klambda-1)!}   \\
0 & 1  & x & \cdots & \frac{x^{\klambda-2} }{(\klambda-2)!}  \\
0 & 0  & 1  & \ddots & \vdots \\
\vdots & \ddots    & \ddots & \ddots & x  \\
0  & \cdots &   & 0  &1 
\end{pmatrix}.	\label{eq:exp matrix}
\end{align}
We have
\begin{equation}	\label{eq:exp(lambda,x)exp(lambda,y)=exp(lambda,x+y)}
\exp(\lambda, x)  \exp(\lambda,y)=\exp(\lambda,x+y),	\quad	x,y \in \R.
\end{equation}
Writing $J_\lambda$ for the
$\klambda \times \klambda$ Jordan matrix
\begin{equation*}
J_{\lambda} =
\begin{pmatrix}
\lambda	& 1	& 0			& \ldots	& 0	\\
0	& \lambda		& 1		& \ldots	& 0	\\
\vdots	& 	& \ddots		& \ddots	& \vdots	\\
0	& \ldots	& 	& \lambda	& 1	\\
0	& \ldots	& 	& 	& \lambda	\\
\end{pmatrix}
\end{equation*}
we infer \eqref{eq:exp(lambda,x)exp(lambda,y)=exp(lambda,x+y)} from $\exp(\lambda,x) = \exp(x J_\lambda)$
from the functional equation of the matrix exponential function.
\end{appendix}


%
\bibliographystyle{abbrv}



\end{document}